\documentclass[12pt]{article}
\usepackage[utf8]{inputenc}
\usepackage{amsmath}
\usepackage{amsfonts}
\usepackage{amssymb}
\usepackage{amsthm}
\usepackage[dvipsnames]{xcolor}
\usepackage{comment}
\usepackage{algorithm}
\usepackage{algorithmic}
\usepackage{mathtools}
\usepackage{abraces}

\newcommand{\C}{\mathbb{C}}
\newcommand{\R}{\mathbb{R}}

\newcommand{\Bdist}{\mathrm{Beta}}
\newcommand{\Bfun}{\mathrm{B}}
\newcommand{\dx}{\,\mathrm{d}x}
\newcommand{\dy}{\,\mathrm{d}y}
\newcommand{\dz}{\,\mathrm{d}z}

\newcommand\change[1]{{{\color{black}#1}}}

\usepackage{eqparbox}

\usepackage{geometry}
\DeclareMathOperator{\nrank}{nrank}
\DeclareMathOperator{\rank}{rank}
\DeclareMathOperator{\diag}{diag}

\DeclareMathOperator{\rev}{rev}

\newtheorem{theorem}{Theorem}[section]
\newtheorem{definition}[theorem]{Definition}

\newtheorem{proposition}[theorem]{Proposition}
\newtheorem{lemma}[theorem]{Lemma}
\newtheorem{remark}[theorem]{Remark}
\newtheorem{example}[theorem]{Example}

\title{Singular quadratic eigenvalue problems: \\ Linearization and weak condition numbers\footnote{The work of the first author was supported by the SNSF research project \emph{Probabilistic methods for joint and singular eigenvalue problems}, grant number: 200021L\_192049. The work of the second author was supported by the CSF research project \emph{Randomized low rank algorithms and applications to parameter dependent problems}, grant number: IP-2019-04-6268. Part of this work was
done while the second author was a PostDoctoral researcher at EPFL.}}
\date{}
\author{
  Daniel Kressner\footnote{Institute of Mathematics, EPFL, CH-1015 Lausanne, Switzerland. E-mail: daniel.kressner@epfl.ch.} 
  \and
  Ivana \v{S}ain Glibi\'{c}\footnote{Faculty of Science, Department of Mathematics, University of Zagreb, HR-10000 Zagreb, Croatia. E-mail: ivanasai@math.hr.}
}

\begin{document}

\maketitle
\begin{abstract}
The numerical solution of singular eigenvalue problems is complicated by the fact that small perturbations of the coefficients may have an arbitrarily bad effect on eigenvalue accuracy. However, it has been known for a long time that such perturbations are exceptional and standard eigenvalue solvers, such as the QZ algorithm, tend to yield good accuracy despite the inevitable presence of roundoff error.
Recently, Lotz and Noferini quantified this phenomenon by introducing the concept of $\delta$-weak eigenvalue condition numbers. In this work, we consider singular quadratic eigenvalue problems and two popular linearizations. Our results show that a correctly chosen linearization increases $\delta$-weak eigenvalue condition numbers only marginally, justifying the use of these linearizations in numerical solvers also in the singular case. We propose a very simple but often effective algorithm for computing well-conditioned eigenvalues of a singular quadratic eigenvalue problems by adding small random perturbations to the coefficients. We prove that the eigenvalue condition number is, with high probability, a reliable criterion for detecting and excluding spurious eigenvalues created from the singular part.
\end{abstract}

\section{Introduction}

This work is concerned with the \emph{quadratic eigenvalue problem} associated with a matrix polynomial of the form
\begin{equation}\label{eq:qep}
    Q(\lambda) = \lambda^2M + \lambda C + K, 
\end{equation}
where $M, C, K \in \mathbb{C}^{n\times n}$. We will study the numerical computation of eigenvalues when $Q(\lambda)$ is \emph{singular}, that is, $\det (Q(\lambda)) \equiv 0$.
Quadratic eigenvalue problems occur in a wide variety of applications in particular in mechanical and electrical engineering; see~\cite{Tisseur2001} for an overview. More specifically, differential algebraic equations and control problems frequently give rise to
singular (linear or quadratic) eigenvalue problems; see, e.g.,~\cite{byers2008trimmed,Shi2004,VanDooren1981}. To define eigenvalues for such problems, one considers the \emph{normal rank} of $Q$ defined as \[ \nrank(Q) := \max_{\lambda \in\mathbb{C}}\rank(Q(\lambda)).\] Then 
$\lambda_0 \in \C$ is called a (finite) eigenvalue of $Q$ if the rank of $Q(\lambda_0)$ is strictly smaller than $\nrank(Q)$.

The numerical solution of singular eigenvalue problems is complicated by the existence of arbitrarily small matrix perturbations that move eigenvalues anywhere in the complex plane~\cite{kaagstrom20008}. Therefore, it may seem futile to attempt to compute such eigenvalues in finite precision arithmetic. On the other hand, it was recognized early on by Wilkinson~\cite{wilkinson1979kronecker} that the QZ algorithm~\cite{Golub2013} applied to a singular linear eigenvalue problem usually returns good eigenvalue approximations, despite the presence of perturbations introduced by roundoff error.
More recently, De Ter\'an, Dopico, and Moro~\cite{de2008first} explained Wilkinson's observation by proving that perturbation directions causing arbitrarily large eigenvalue changes are rare; in fact, they form a set of measure zero. For the other perturbation directions, the eigenvalue sensitivity has been analyzed in~\cite{de2008first} and~\cite{de2010first} for linear and polynomial eigenvalue problems, respectively. These results are of a qualitative nature; in particular, they do not allow to draw conclusions on the eigenvalue error one should normally \emph{expect}. Building on~\cite{de2008first,de2010first}, Lotz and Noferini~\cite{lotz2020wilkinson} address this question by assuming that the perturbation directions are uniformly distributed on the sphere and deriving the expected value as well as a tail bound for the eigenvalue sensitivity. The latter yields an upper bound on the so called \emph{$\delta$-weak (stochastic) condition number}, a quantity introduced in~\cite{lotz2020wilkinson} that bounds the eigenvalue sensitivity with probability $1-\delta$ for prescribed $0 \le \delta < 1$.

\emph{Linearization} is often the first step of numerical solvers for a quadratic eigenvalue problem, which usually means turning it into a linear eigenvalue problem of double size. Various different linearizations are available in the literature; the seminal work by Mackey et al.~\cite{mackey2006vector} identifies two useful vector spaces of linearizations, $\mathbb{L}_1$ and $\mathbb{L}_2$, which include the popular companion linearizations. The choice of linearization influences the reliability and accuracy of the eigenvalue solver. For regular matrix polynomials, Higham et al.~\cite{higham2006conditioning} have analyzed the effect of linearization on the eigenvalue condition numbers for linearizations in $\mathbb{DL} = \mathbb{L}_1\cap \mathbb{L}_2$ as well as companion linearizations. For a \emph{singular} matrix polynomial $P(\lambda)$, 
De T\'eran, Dopico and Mackey~\cite{de2009linearizations} showed that almost all elements from $\mathbb{L}_1$ and $\mathbb{L}_2$ still satisfy the usual definition of (strong) linearization. Additionally, they also explain how the so called minimal indices and minimal bases of $P$ are related to the minimal indices and minimal bases of the corresponding linearization. Interestingly, $\mathbb{DL}$ does \emph{not} contain a linearization for singular matrix polynomials\change{; see~\cite{Dopico2022} for recent work how key quantities can still be recoverd from linearizations in $\mathbb{DL}$.}

A common approach to solving a singular linear eigenvalue problem $C(\lambda) = A-\lambda B$ is to first extract the regular part by first computing the so called staircase form~\cite{Demmel1993,Demmel1993a,VanDooren1979} and then applying the QZ algorithm. For singular polynomial eigenvalue problems, this extraction can simply be applied after linearization; exploiting the result from~\cite{de2009linearizations} discussed above. A more sophisticated technique of combing extraction with linearization is derived in~\cite{byers2008trimmed}. As discussed, e.g., in~\cite[Pg. 1023]{hochstenbach2019solving}, the computation of the staircase form can be significantly more expensive than the QZ algorithm and, possibly more importantly, the identification of the regular part crucially relies on rank decisions that are prone to become incorrect in finite precision arithmetic. In view of Wilkinson's observation, one may wonder whether it is not possible to skip the extraction step and apply QZ directly to $C(\lambda)$.  One challenge, certainly from a theoretical perspective but quite possibly also from a practical perspective, is that it is unreasonable to use standard random matrix distributions for modeling the roundoff error incurred by the QZ algorithm; see~\cite[Sec. 2.8]{Higham2002}. One way to fix this is to introduce prescribed random perturbations in $A$ and $B$. Hochstenbach, Mehl, and Plestenjak~\cite{hochstenbach2019solving} proposed to use random perturbations of rank $k := n-\nrank(C)$, which -- with probability one -- leave the eigenvalues of $C(\lambda)$ unchanged and create
additional spurious eigenvalues from the perturbed singular part. Heuristic yet effective criteria are used to identify and exclude these spurious eigenvalues. From a theoretical perspective, a major disadvantage of such rank-$k$ perturbations is that they are much harder to analyze than unstructured random perturbations; they certainly do not fit the framework of~\cite{lotz2020wilkinson}. In turn, the effect of such perturbations on weak eigenvalue condition numbers remains unclear and it is difficult to develop a theoretically justified criterion for identifying spurious eigenvalues.

This work makes the following new theoretical contributions to the treatment of singular quadratic eigenvalue problems:
\begin{enumerate}
\item[(C1)] Theorems~\ref{t:FCFratio} and~\ref{t:ACFratio} bound the effect of companion linearizations on $\delta$-weak eigenvalue condition numbers; thereby extending some of the results from~\cite{higham2006conditioning} to the singular case.
\item[(C2)] Section~\ref{sec:boundcond} develops a theoretical criterion for correctly identifying, with high probability, all simple well-conditioned eigenvalues of a matrix polynomial subject to uniformly distributed perturbations.
 \end{enumerate}
To achieve these contributions, our work closely follows and extends the work by Lotz and Noferini~\cite{lotz2020wilkinson}. For example, for (C1) the upper bounds on $\delta$-weak eigenvalue condition numbers from~\cite{lotz2020wilkinson} are combined with newly developed lower bounds; see Theorem~\ref{t:deltaWeakCondLowerBound}. Both, (C1) and (C2) have immediate practical consequences. In particular, we show that there is always a numerically safe choice of companion linearization, that is, the linearization only has a mild impact on condition numbers. Note, however, that one may have to choose a different companion linearization depending on whether the eigenvalues of interest have small or large magnitude, which is in line with similar results~\cite{higham2006conditioning,adhikari2011structured} for regular matrix polynomials.

\section{Definitions and preliminaries}

In order to introduce notation and definitions for linear and quadratic polynomials in a uniform way \change{and maintain generality as much as possible}, we first consider a general matrix polynomial taking the form
\begin{equation} \label{eq:genpoly} P(\lambda) = \sum^m_{i=0}\lambda^iA_i,\quad A_i\in\mathbb{C}^{n\times n}.
\end{equation}
  We recall that $\lambda_0$ is a (finite) eigenvalue of $P$ if
$
    \rank(P(\lambda_0)) < \nrank (P) := \max_{\lambda \in\mathbb{C}}\rank(P(\lambda)).
$

In this work, we will only deal with simple eigenvalues; if $P$ is singular the Smith canonical form can be used to define when an eigenvalue is simple~\cite{de2010first}.

To define eigenvectors for singular $P$, we follow~\cite{Dopico2020,lotz2020wilkinson}. Let $\mathbb{C}(\lambda)$ denote the field of rational functions with complex coefficients and $\mathbb{C}^n(\lambda)$ the vector space over $\mathbb{C}(\lambda)$ of $n$--tuples of rational functions.
\begin{definition}[\cite{de2009linearizations}] \label{d:nullspaces}
For an $n\times n$ matrix polynomial $P(\lambda)$, we call 
\[
 \mathcal{N}(P) : = \left\{ x(\lambda) \in \mathbb{C}(\lambda)^n:\; P(\lambda)x(\lambda) \equiv 0 \right\}
\]
and
\[
 \mathcal{N}_*(P) : = \left\{ y(\lambda) \in \mathbb{C}\big(\overline{\lambda}\big)^n:\; y(\lambda)^* P(\lambda) \equiv 0 \right\}
\]
the \emph{right} and \emph{left nullspace} of $P$, respectively.
\end{definition}
The two subspaces have the same dimension because \[ r:=\nrank(P) = n-\dim \mathcal{N}(P) = n-\dim \mathcal{N}_*(P). \]
To work with these subspaces, we first note that one can always choose polynomial bases. In particular, an $n\times (n-r)$ matrix polynomial
$\big[ x_1(\lambda),\ldots,x_{n-r}(\lambda)\big]$ is called a \emph{minimal basis} of $\mathcal{N}(P)$ if the sum of the degree of its columns is minimal among all polynomial bases of $\mathcal{N}(P)$. Analogously, we let 
$\big[ y_1(\lambda),\ldots,y_{n-r}(\lambda)\big]$ denote a minimal basis of $\mathcal{N}_*(P)$.
Then \change{one defines~\cite{Dopico2020} for a finite eigenvalue $\lambda_0$ the subspaces}
\begin{align*}
\begin{split}
\ker_{\lambda_0}(P) &:= \mathrm{range} \big[ x_1(\lambda_0),\ldots,x_{n-r}(\lambda_0)\big] \subsetneq \ker(P(\lambda_0)),  \\
    \ker_{\lambda_0}(P^*) &:= \mathrm{range} \big[ y_1(\lambda_0),\ldots,y_{n-r}(\lambda_0)\big] \subsetneq \ker(P(\lambda_0)^*).
\end{split}
\end{align*}
In accordance with~\cite{de2010first}, $\ker_{\lambda_0}(P)$ and $\ker_{\lambda_0}(P^*)$ are called \emph{right} and \emph{left singular space}, respectively.
\begin{definition} \label{def:eigenvector}
 Let $\lambda_0$ be an eigenvalue of an $n\times n$ matrix polynomial $P(\lambda)$. A vector $x$ is called a right eigenvector of $P$ \change{if $x \in \ker( P(\lambda_0) )$ and
 $x\not\in \ker_{\lambda_0}(P)$.}
  A vector $y$ is called a left eigenvector of $P$ if \change{$y \in \ker( P(\lambda_0)^* )$ and $y\not\in \ker_{\lambda_0}(P^*)$.} 
\end{definition}

Let us now suppose that $\lambda_0$ is simple, which implies $\dim \ker(P(\lambda_0)) = n-r+1$, and consider any basis $X\in\mathbb{C}^{n\times(n-r)}$ of $\ker_{\lambda_0}(P)$. Then $x$ is a right eigenvector if and only if
$\begin{bmatrix} X & x \end{bmatrix} \in \mathbb{C}^{n\times (n-r+1)}$ is a basis for $\ker (P(\lambda_0))$. Given a basis $Y\in\mathbb{C}^{n\times(n-r)}$ of $\ker_{\lambda_0}(P^*)$, $y$ is a left eigenvector if and only if
$\begin{bmatrix} Y & y \end{bmatrix} \in \mathbb{C}^{n\times (n-r+1)}$ is a basis for $\ker (P(\lambda_0)^*)$. Note that, compared to simple eigenvalues of regular eigenvalue problems, there is much more freedom in choosing eigenvectors, e.g., $x$ is only determined up to components from $\ker_{\lambda_0}(P)$.

\subsection{$\delta$--weak eigenvalue condition number}\label{SectionDeltaWeak}

We now perturb $P$ with another $n\times n$ matrix polynomial $E(\lambda) = \sum^m_{i=0}\lambda^iE_i$,
\begin{equation}\label{eq:pertPEP}
    \tilde{P}(\lambda) = P(\lambda)+\varepsilon E(\lambda),\;\varepsilon>0.
\end{equation}
The norm of $E$ is measured by
\[
    \|E\|^2 := \big\| \begin{bmatrix}
    E_0 & E_1 & \cdots &E_m
\end{bmatrix} \big\|_F^2 = \sum^m_{i=0}\|E_i\|^2_F,
\]
where $\|\cdot\|_F$ denotes the Frobenius norm of a matrix.
In the singular case, the effect of perturbations on the accuracy of eigenvalues can be arbitrarily bad. To demonstrate this, we use a variation of an example from~\cite{kaagstrom20008}.
\begin{example}
    Consider the singular quadratic eigenvalue problem
    \begin{equation*}
        Q(\lambda) = \lambda^2\begin{bmatrix}
            1 & 0\\
            0 & 0
        \end{bmatrix} + \lambda \begin{bmatrix}
            -3 & 0\\
            0 & 0
        \end{bmatrix} + \begin{bmatrix}
            2 & 0\\
            0 & 0
        \end{bmatrix},
    \end{equation*}
    with two simple eigenvalues $1$ and $2$. The perturbed polynomial
    \begin{equation*}
        \tilde{Q}(\lambda) = \lambda^2\begin{bmatrix}
            1 & \varepsilon_1\\
            \varepsilon_2 & 0
        \end{bmatrix} + \lambda \begin{bmatrix}
            -3 & \varepsilon_3\\
            \varepsilon_4 & 0
        \end{bmatrix} + \begin{bmatrix}
            2 & \varepsilon_5\\
            \varepsilon_6 & 0
        \end{bmatrix}
    \end{equation*}
    is regular when $\varepsilon_i \not = 0$ and its eigenvalues are the zeros of \[ (\varepsilon_1\lambda^2 + \varepsilon_3\lambda + \varepsilon_5)(\varepsilon_2\lambda^2+\varepsilon_1\lambda+\varepsilon_6).\] By choosing the ratios
    $\varepsilon_3/\varepsilon_1$, $\varepsilon_3/\varepsilon_5$, $\varepsilon_1/\varepsilon_2$, $\varepsilon_6/\varepsilon_2$ appropriately, these perturbed eigenvalues can be placed anywhere in the complex \change{plane} for arbitrarily small $\varepsilon_i$. As shown in~\cite{de2010first} and easily verified for $Q$, this perturbation constitutes an exception. For generic $E$, the perturbed polynomial $\tilde Q(\lambda) = Q(\lambda) + \varepsilon E(\lambda)$ has eigenvalues close to $1,2$ for $\varepsilon>0$ sufficiently small.
\end{example}

To quantify the effect of perturbations in the generic case, let $\begin{bmatrix}
    X & x
\end{bmatrix}$ and $\begin{bmatrix}
    Y & y
\end{bmatrix}$ denote the bases for \change{$\ker \left(P(\lambda_0)\right)$} and \change{$\ker \left(P(\lambda_0)\right)^*$} introduced above for a simple eigenvalue $\lambda_0$ of $P$. Assume that $Y^*E(\lambda_0) X$ is nonsingular, which is generically satisfied for $E(\lambda)$. The first-order eigenvalue perturbation expansion by De Ter\'an, and Dopico~\cite{de2010first} states that there is an eigenvalue $\lambda_0(\varepsilon)$ of the perturbed polynomial $\tilde{P}(\lambda)$ such that
\begin{equation}\label{eq:Expansion}
    \lambda_0(\varepsilon) = \lambda_0 - \frac{\det (\begin{bmatrix} Y & y \end{bmatrix}^*E(\lambda_0)\begin{bmatrix} X & x \end{bmatrix})}{y^*P'(\lambda_0)x\cdot \det(Y^*E(\lambda_0)X)}\varepsilon + O(\varepsilon^2).
\end{equation}

It follows that the \emph{directional sensitivity} of $\lambda_0$ satisfies
\begin{align}
    \sigma_E &:= \lim_{\varepsilon\to 0}\frac{|\lambda_0(\varepsilon) - \lambda_0|}{\varepsilon \|E\|} \nonumber  \\
    &= \frac{1}{\|E\|}\left| \frac{\det (\begin{bmatrix} Y & y \end{bmatrix}^*E(\lambda_0)\begin{bmatrix} X & x \end{bmatrix})}{y^*P'(\lambda_0)x\cdot \det(Y^*E(\lambda_0)X)} \right|. \label{sigmaEformula}
\end{align}
The main difficulty in continuing from here is that $\sigma_E$ becomes arbitrarily large when $Y^*E(\lambda_0) X$ is close to a singular matrix. Still, one hopes that the set of such bad perturbation directions, for which $\sigma_E$ becomes large, remains small. In order to quantify this, one assumes that the matrix $\begin{bmatrix}
    E_0 & E_1 & \cdots &E_m
\end{bmatrix}$ defining the perturbation is randomly distributed, more precisely its vectorization is uniformly distributed over the unit sphere in $\C^{n^2(m+1)}$. In the following, this will be denoted by
\[
 E\sim \mathcal{U}(2N), \quad N=n^2(m+1).
\]
Lotz and Noferini~\cite{lotz2020wilkinson} have defined the $\delta$--weak condition number as the smallest bound $\kappa_w(\delta)$ such that the $\sigma_E$ stays below $\kappa_w(\delta)$ with probability at least $1-\delta$.
\begin{definition}
Let $\lambda_0$ be a simple eigenvalue of an $n\times n$ matrix polynomial $P$. Then 
the $\delta$--weak condition number of $\lambda_0$ is defined as
\[
    \kappa_w(\delta) = \inf \{t\in \mathbb{R}: \mathbb{P}(\sigma_E < t)\geq 1-\delta \},
\]
\change{where} $\sigma_E$ is defined as in~\eqref{sigmaEformula} and $E\sim \mathcal{U}(2N)$.
\end{definition}
A crucial step in deriving bounds for $\kappa_w(\delta)$ is to exploit the unitary invariance of (complex) Gaussian random matrices. More specifically, since $\sigma_E$ is invariant under scaling of $E$, we can remove the normalization of $E$ and assume that the entries of its coefficients are i.i.d. normal random variables with mean $0$ and variance $\big( \sum^m_{j=0}|\lambda_0|^{2j} \big)^{-1}$. In turn, $E(\lambda_0)$ is a complex Gaussian random matrix. \emph{If} the columns of $\begin{bmatrix}
    X & x
\end{bmatrix}$ and $\begin{bmatrix}
    Y & y
\end{bmatrix}$ are orthonormal, the invariance of complex Gaussian random matrices under multiplication with unitary matrices implies that we can choose w.l.o.g. bases of unit vectors and hence 
\begin{equation*}
    \frac{\left| \det \left(\begin{bmatrix} Y & y \end{bmatrix}^*E(\lambda_0)\begin{bmatrix} X & x \end{bmatrix}\right) \right|}{\left| \det(Y^*E\left(\lambda_0)X\right) \right|} \sim \frac{\left| \det (G) \right|}{\left| \det(\overline{G}) \right|} \sim \frac{1}{|h_{\ell \ell}|}, \;\; \ell = n-r+1,
\end{equation*}
where $G$ is an $\ell \times \ell$ complex Gaussian matrix, $\overline{G}$ is the leading principal $(\ell-1)\times (\ell-1)$ submatrix of $G$, and $H = [h_{ij}]_{i,j=1}^n=G^{-1}$. Setting
\begin{equation}\label{eq:gammaP}
    \gamma_P = |y^*P'(\lambda_0)x| \change{/ \|\Lambda^m_0\|_2}
\end{equation}
\change{with
\begin{equation}\label{eq:lambda0}
 \Lambda_0^{m} = [\lambda_0^{m},\lambda_0^{m-1},\ldots,\lambda_0,1]^T,
\end{equation}}%
it follows from~\eqref{sigmaEformula} that
\[
 \mathbb{P}(\sigma_E\geq t) = \mathbb{P}\Big(  |h_{\ell\ell}|^{-2} \ge \gamma^2_P t^2 \|E\|^2 \change{\|\Lambda^m_0\|_2^2}  \Big).
\]
Deriving an explicit expression for the distribution of $|h_{\ell\ell}|^{-1}$, the following result was obtained in~\cite{lotz2020wilkinson}.
\begin{theorem} \label{t:sigmaEdistr}
Let $\lambda_0$ be a simple eigenvalue of an $n\times n$ matrix polynomial $P$ with normal rank $r$. Let $\gamma_P$ be defined as in~\eqref{eq:gammaP}. If $E\sim \mathcal{U}(2N)$ the directional sensitivity of $\lambda_0$ satisfies
\[
        \mathbb{P}(\sigma_E\geq t) = \mathbb{P}(Z_N/Z_{n-r+1}\geq \gamma_P^2t^2),
\]
    where $Z_k\sim \Bdist(1,k-1)$ denotes a beta-distributed random variable with parameters $1$ and $k-1$, and $Z_N$ and $Z_{n-r+1}$ are independent.
\end{theorem}
Using the result of Theorem~\ref{t:sigmaEdistr}, Lotz and Noferini~\cite{lotz2020wilkinson} compute the expected value of $\sigma_E$ and derive the tail bound    

$$\mathbb{P}(\sigma_E\geq t) \leq \frac{1}{\gamma^2_P}\frac{n-r}{N}\frac{1}{t^2}$$
\change{for} $t\geq \gamma_P^{-1}$. In turn, this yields the upper bound
\begin{equation}\label{eq:deltaWeakCondUpperBound}
\kappa_w(\delta) \leq \frac{1}{\gamma_P} \max \left\{  1,\sqrt{\frac{n-r}{\delta N}} \right\};
\end{equation}
see~\cite[Theorem 5.1]{lotz2020wilkinson}.
\begin{remark}
For a polynomial eigenvalue problem with real coefficient matrices, it would be more natural to work with  real perturbations. If $\lambda_0$ is real then $\begin{bmatrix}
    X & x
\end{bmatrix}$ and $\begin{bmatrix}
    Y & y
\end{bmatrix}$ can be chosen real and the discussion above can be easily adjusted to treat the case that $E$ is uniformly distributed on the unit sphere in $\R^{n^2(m+1)}$; the corresponding result is Theorem 5.2 in~\cite{lotz2020wilkinson}.

If $\lambda_0$ is complex then 
$\begin{bmatrix}
    X & x
\end{bmatrix}$ and $\begin{bmatrix}
    Y & y
\end{bmatrix}$ can \emph{not} be chosen real. As real Gaussian random matrices are clearly not invariant under (complex) unitary transformations, the arguments given above do not extend to this situation.
In turn, the proof (and possibly the results themselves) of Theorem 5.2 and related statements in~\cite{lotz2020wilkinson} are incorrect. For regular eigenvalue problems, it is known~\cite{Byers2004a} that restricting the perturbations from complex to real does not have a significant impact on the condition number. It remains open whether such a result can be established for weak condition numbers.

\end{remark}

\subsection{Linearization of singular \change{polynomial} eigenvalue problems} \label{sec:linearization}

As explained in the introduction, it is common to apply linearization when numerically solving polynomial eigenvalue problem.
\begin{definition}
A linear matrix polynomial (matrix pencil) $L(\lambda) = \lambda L_1 + L_0$ with $L_0,L_1 \in\mathbb{C}^{nm\times nm}$ is a \emph{linearization} of an $n\times n$ matrix polynomial $P(\lambda) = \sum^m_{i=0}\lambda^iA_i$ if there exist unimodular $nm\times nm$ matrices $E(\lambda)$ and $F(\lambda)$ such that
\[
    E(\lambda)L(\lambda)F(\lambda) = \begin{bmatrix}
    P(\lambda) & 0\\
    0 & I_{(m-1)n}
    \end{bmatrix}.
\]
\change{A} linearization $L(\lambda)$ is called a \emph{strong linearization} if $\rev L(\lambda) = \change{\lambda L_0 + L_1}$ is also a linearization of $\rev P(\lambda) = \sum^m_{i=0}\lambda^iA_{\change{m}-i}$.
\end{definition}
As emphasized in~\cite{de2009linearizations},
only strong linearizations \change{preserve} the complete eigenstructure of a singular matrix polynomial
and it is therefore preferable to work with them.
\change{From now on, we will focus on the \emph{first companion form linearization}
\begin{equation}\label{eq:c1general}
    C_1(\lambda) = \lambda \begin{bmatrix} A_m & && 0 \\  & I_n \\ && \ddots \\ 0 &&& I_n \end{bmatrix} + \begin{bmatrix} A_{m-1} & A_{m-2} & \cdots & A_0 \\ -I_n & 0 & \cdots & 0 \\  & \ddots & \ddots & \vdots \\ 0 &  & -I_n & 0 \end{bmatrix},
\end{equation}}%
which is known to be a strong linearization for both regular and singular matrix polynomials~ \cite{gohberg1988general,de2009linearizations}.
\change{For a} (singular) quadratic matrix polynomial
$Q(\lambda) = \lambda^2M + \lambda C + K$\change{, this takes the form}
\begin{equation}\label{eq:c1quad}
    C_1(\lambda) = \lambda \begin{bmatrix} M & 0\\ 0 & I_n \end{bmatrix} + \begin{bmatrix} C & K\\ -I_n & 0 \end{bmatrix},
\end{equation}

The recovery of eigenvectors of a (regular) matrix polynomial from the eigenvectors of its linearization has been systematically studied in~\cite{mackey2006vector}. Extending these results to the singular case, the recovery of minimal bases is discussed in~\cite{de2009linearizations}. Applied to the linearization $C_1(\lambda)$ of \change{$P(\lambda)$}, we obtain the following results.

\begin{paragraph}{Minimal basis of $\mathcal{N}(C_1)$ and right eigenvector.}
Given a minimal basis $\big[ x_1(\lambda),\ldots,x_{n-r}(\lambda)\big]$ of $\mathcal{N}(\change{P})$, Theorem 5.2 in~\cite{de2009linearizations} states that
\[\change{\begin{bmatrix}
 \lambda^{m-1} x_1(\lambda) & \ldots & \lambda^{m-1} x_{n-r}(\lambda) \\
\vdots & & \vdots \\
 \lambda x_1(\lambda) & \ldots & \lambda x_{n-r}(\lambda) \\
 x_1(\lambda) & \ldots & x_{n-r}(\lambda) 
\end{bmatrix}}
\]
is a \emph{minimal basis} of $\mathcal{N}(C_1)$. Setting $X = \big[ x_1(\lambda_0),\ldots,x_{n-r}(\lambda_0)\big]$ for a simple eigenvalue $\lambda_0$ of \change{$P$}, we thus obtain that \change{$\Lambda_0^{m-1}\otimes X$} is a basis of $\ker_{\lambda_0}(C_1)$, \change{where the vector $\Lambda_0^{m-1}$ is defined as in~\eqref{eq:lambda0} and `$\otimes$' denotes the Kronecker product}. Now, if $x$ is a right eigenvector of \change{$P$} in the sense of Definition~\ref{def:eigenvector} one easily verifies that $v := \change{\Lambda_0^{m-1} \otimes x}$
satisfies $C_1(\lambda_0) v = 0$ and $v \not\in \ker_{\lambda_0}(C_1)$. In other words, $v$ is a right eigenvector of $C_1$. To summarize,
\begin{equation}\label{eq:rightC1norm1}
    \begin{bmatrix}  X_L & x_L \end{bmatrix} := \change{ \big( \Lambda_0^{m-1} \otimes \begin{bmatrix} X & x  \end{bmatrix} \big) / \|\Lambda_0^{m-1}\|_2}
\end{equation}
is a basis of $\ker(C_1(\lambda_0))$ with the first $n-r$ columns being a basis of $\ker_{\lambda_0}(C_1)$. For later purposes, the scaling is chosen such that $\begin{bmatrix}  X_L & x_L \end{bmatrix}$ becomes orthonormal if 
$\begin{bmatrix} X  & x \end{bmatrix}$ is orthonormal.
\end{paragraph}

\begin{paragraph}{Minimal basis of $\mathcal{N}(C_1^*)$ and left eigenvector.}
 Let $\big[ y_1(\lambda),\ldots,y_{n-r}(\lambda)\big]$ be a minimal basis of $\mathcal{N}(\change{P}^*)$.
 Then 
 \[
  C_1(\lambda)^* \change{\mathcal P(\lambda)^*y_i(\lambda)}= 0
 \]
 for $i = 1,\ldots,n-r$, with the $n\times mn$ matrix $\mathcal P(\lambda) = [P_0(\lambda),\ldots, P_{m-1}(\lambda) ]$ defined by
 \[
  \change{P_0(\lambda) = I_n, \quad P_j(\lambda) = \lambda^j A_m + \lambda^{j-1} A_{m-1} + \cdots + A_{m-j}, \quad j = 1,\ldots,m-1. }
 \]
Together with Theorem 5.7 in~\cite{de2009linearizations} this implies that
$
\change{\mathcal P(\lambda)^*\big[ y_1(\lambda),\ldots,y_{n-r}(\lambda)\big]}
$
is a minimal basis of $\mathcal{N}(C_1^*)$.
Similarly, it follows that if $y$ is a left eigenvector of \change{$P$} associated with $\lambda_0$ then 
\change{$\mathcal P(\lambda_0)^*y$}
is a left eigenvector of $C_1$; \change{see also~\cite[Lemma 7.2]{higham2006conditioning}}.
Letting
$Y = \big[ y_1(\lambda_0),\ldots,y_{n-r}(\lambda_0)\big]$ we have thus obtained that 
\begin{equation} \label{eq:LeftC1}
 \begin{bmatrix}
   \tilde Y_L & \tilde y_L
   \end{bmatrix} = \change{\mathcal P(\lambda_0)^* 
 \begin{bmatrix} Y & y \end{bmatrix}},
\end{equation}
is a basis of $\ker(C_1(\lambda_0)^*)$ with the first $n-r$ columns being a basis of $\ker_{\lambda_0}(C_1^*)$. If $\begin{bmatrix} Y  & y \end{bmatrix}$ is orthonormal, an orthonormal basis $\begin{bmatrix} Y_L  & y_L \end{bmatrix}$ of $\ker(C_1(\lambda_0)^*)$ with the same property is obtained by setting
\begin{equation}  \label{eq:LeftC1orth}
 Y_L:= \tilde Y_L (\tilde Y_L^* \tilde Y_L)^{-1/2} = \change{\mathcal P(\lambda_0)^*Y}S, \quad 
y_L:= \frac{1}{\beta} \Pi_\perp\tilde y_L = \frac{1}{\beta} \Pi_\perp \change{\mathcal P(\lambda_0)^* y},
\end{equation}
with 
the matrix $
            \change{S = \big( I + Y^* P_1(\lambda) P_1(\lambda)^* Y + \cdots + Y^* P_{m-1}(\lambda) P_{m-1}(\lambda)^* Y \big)^{-1/2},}
           $
the orthogonal projector $\Pi_\perp := I - Y_L Y_L^*$ onto $\ker_{\lambda_0}(C_1^*)^\perp$, and
\change{\begin{equation}\label{eq:beta}
\beta = \|\Pi_\perp\tilde y_L\|_2.
\end{equation}}
\end{paragraph}

\section{$\delta$--weak condition number of the linearization}\label{s:DeltaCNLin}

In this section we quantify the impact of the linearization \eqref{eq:c1general} on $\delta$--weak eigenvalue condition numbers. More precisely, if $\kappa^{(\change{P})}_w(\delta)$ and
$\kappa^{(C_1)}_w(\delta)$ denote the $\delta$--weak condition numbers of a simple eigenvalue $\lambda_0$ for \change{a polynomial eigenvalue problem $P$ and its} linearization~\eqref{eq:c1general}, respectively, we aim at determining an upper bound on the ratio
\begin{equation} \label{eq:ratio}
    \frac{\kappa^{(C_1)}_w(\delta)}{\kappa^{(\change{P})}_w(\delta)}.
\end{equation}
For regular matrix polynomials, this question was first addressed in~\cite{higham2006conditioning}, followed up by a more detailed study of the quadratic case in~\cite{zeng2014backward}. 

To address the singular case, we face two complications. First, the relation between the eigenvectors of $C_1$ and $\change{P}$ is more subtle due to the presence of singular spaces; see, e.g.,~\eqref{eq:LeftC1orth}. 
Second, in order to get an upper bound for~\eqref{eq:ratio} we also need a \emph{lower} bound for the denominator $\kappa^{(\change{P})}_w(\delta)$, which requires an extension of the analysis in~\cite{lotz2020wilkinson}.

\subsection{Lower bound for $\delta$--weak condition number}

In this section, we establish a lower bound for $\delta$--weak eigenvalue condition numbers.
\begin{theorem}\label{t:deltaWeakCondLowerBound}
Let $\lambda_0$ be a simple eigenvalue of an $n\times n$ matrix polynomial $P$ of the form~\eqref{eq:genpoly} with normal rank $r<n$, let $\gamma_P$ be defined as in~\eqref{eq:gammaP}, and set $N = n^2(m+1)$. Then the $\delta$--weak eigenvalue condition number of $\lambda_0$ satisfies
\begin{equation}\label{eq:deltaWeakCondLowerBound}
    \kappa^{\change{(P)}}_w(\delta) \geq \sqrt{\frac{(N-1)(n-r)}{(N+n-r-2)(N+n-r-1)\delta}}\frac{1}{\gamma_P} \ge \frac{1}{\sqrt{N\delta}} \frac{1}{\gamma_P},
\end{equation}
for $\delta \leq \frac{(N-1)(n-r)}{(N+n-r-2)(N+n-r-1)}$.
\end{theorem}
In order to prove Theorem~\ref{t:deltaWeakCondLowerBound}, we need the following result, a variation of Lemma A.1 in~\cite{lotz2020wilkinson}.
\begin{lemma}\label{l:TailBound}
For $a,b,c,d>0$ consider independent beta-distributed random variables $X\sim \Bdist(a,b)$, $Y\sim \Bdist(c,d)$. Then
\[
    \mathbb{P}((X/Y)^{1/k}<t)\leq 1-\frac{1}{t^{ck}} \frac{\Bfun(a+c,b+d-1)}{c \Bfun(a,b)B(c,d)},
\]
holds for $t\geq 1$, where $\Bfun(\cdot,\cdot)$ denotes the beta function.
\end{lemma}
\begin{proof}
We recall that the cdf of $\Bdist(a,b)$ is supported on $[0,1]$ and takes the form
\[
 \mathbb{P}(X \le t_0)= \frac{1}{\Bfun(a,b)} \int_0^{t_0} x^{a-1} (1-x)^{b-1}  \dx
\]
Setting $C = \frac{1}{\Bfun(a,b)\Bfun(c,d)}$ we obtain
\begin{align*}
    \mathbb{P}((X/Y)^{1/k}<t) &= 1-C\int^1_0\int^{t^{-k}x}_{0} x^{a-1}(1-x)^{b-1}y^{c-1}(1-y)^{d-1}\dy\dx \\
    & = 1 - \frac{C}{t^{ck}}\int^1_0\int^{1}_0x^{a+c-1}(1-x)^{b-1}z^{c-1}(1-t^{-k}xz)^{d-1}\dz\dx\\
    &\leq 1-\frac{C}{t^{ck}}\int^1_0\int^1_{0}x^{a+c-1}(1-x)^{(b+d-1)-1}z^{c-1}\dz\dx \\
    &= 1-\frac{C}{t^{ck}}\frac{1}{c}\Bfun(a+c,b+d-1) = 1-\frac{1}{t^{ck}}\frac{1}{c}\frac{\Bfun(a+c,b+d-1)}{\Bfun(a,b)\Bfun(c,d)},
\end{align*}
where we substituted $y = t^{-k}x z$ in the second equality
and used that $t\geq 1$ implies $t^{-k}xz\leq x$ in the inequality.
\end{proof}
\begin{proof}[of Theorem~\ref{t:deltaWeakCondLowerBound}]
We recall from Theorem~\ref{t:sigmaEdistr} that
$
 \mathbb{P}(\sigma_E<t) = \mathbb{P}(Z_N/Z_{n-r+1}< \gamma^2_Pt^2),
$
where $Z_k \sim \Bdist(1,k-1)$. Applying Lemma~\ref{l:TailBound} with 
$k=2$, $a = c = 1$, $b=N-1$, $d=n-r$ gives
\[
 \mathbb{P}(\sigma_E<t) \leq 1-\frac{1}{\gamma^2_Pt^2} \frac{\Bfun(2,N+n-r-2)}{\Bfun(1,N-1)\Bfun(1,n-r)} = 1-\frac{1}{\gamma^2_Pt^2} \frac{(N-1)(n-r)}{(N+n-r-2)(N+n-r-1)}
\]
if $\gamma_Pt \ge 1$. Setting $t = t_0:= {\sqrt{\frac{(N-1)(n-r)}{(N+n-r-2)(N+n-r-1)\delta}}}\frac{1}{\gamma_P}$ thus gives $\mathbb{P}(\sigma_E<t_0) \le 1-\delta$, which shows $\kappa^{\change{(P)}}_w(\delta)\geq t_0$ and proves the first inequality in~\eqref{eq:deltaWeakCondLowerBound}. The second inequality is shown by noticing that $t_0$ increases monotonically with $r$
and $t_0 = \frac{1}{\sqrt{N\delta}} \frac{1}{\gamma_P}$ for $r = n-1$.
\end{proof}

\subsection{Bounds for $\delta$--weak condition number ratio}

We now have all the ingredients to state and prove one of the main results of this paper, which bounds the impact of the first companion form linearization on $\delta$--weak condition numbers. \change{The following lemma relates the quantities $\gamma_P$ and $\gamma_{C_1}$, an essential ingredient in relating weak condition numbers.
\begin{lemma}\label{t:FCFratiogeneral}
Let $\lambda_0$ be a simple eigenvalue of a matrix polynomial $P$ taking the form~\eqref{eq:genpoly}, and consider 
the first companion form linearization $C_1$ of $P$. Then the quantities $\gamma_P$ and $\gamma_{C_1}$ defined in~\eqref{eq:gammaP}
satisfy 
\[
 \frac{\gamma_P}{\gamma_{C_1}} = \frac{\beta \|\Lambda_0^1\|_2 \|\Lambda_0^{m-1}\|_2}{\|\Lambda_0^m\|_2},
\]
with $\beta$ and {$\Lambda_0^j$} defined as in~\eqref{eq:beta} and~\eqref{eq:lambda0}, respectively.
\end{lemma}
\begin{proof}
Let $\begin{bmatrix}
    X_L & x_L
\end{bmatrix}$ and $\begin{bmatrix}
    Y_L & y_L
\end{bmatrix}$ be the orthonormal bases of $\ker_{\lambda_0}(C_1)$ and $\ker_{\lambda_0}(C_1^*)$
defined in~\eqref{eq:rightC1norm1} and~\eqref{eq:LeftC1orth}, respectively. 
We proceed as in~\cite{higham2006conditioning} and consider the relation \begin{equation} \label{eq:linrelation}
C_1(\lambda) (\Lambda^{m-1}\otimes I_n) = 
\begin{bmatrix}
P(\lambda) \\
0
\end{bmatrix},\quad  \Lambda^{m-1} = [\lambda^{m-1},\ldots,\lambda,1]^T.
\end{equation}
Differentiating with respect to $\lambda$ at $\lambda_0$ and using that $y_L^* C_1(\lambda) = 0$ gives 
\[
 y_L^* C^\prime_1(\lambda) (\Lambda_0^{m-1}\otimes I_n) =  y_L^* \begin{bmatrix}
P^\prime(\lambda_0) \\
0
\end{bmatrix}.
\]
Multiplying with $x$ and inserting the expressions~\eqref{eq:rightC1norm1},~\eqref{eq:LeftC1orth} for $x_L,y_L$  yields
\begin{equation} \label{eq:expressionlincondnumber}
y_L^*C'_1(\lambda_0)x_L = 
\frac{1}{\beta \|\Lambda_0^{m-1}\|_2} (\Pi_\perp \tilde y_L)^* \begin{bmatrix}
P'(\lambda_0)x \\
0
\end{bmatrix}. 
\end{equation}
To proceed from here, we consider a basis $Y(\lambda)$ of the left null space such that $Y = Y(\lambda_0)$ and set $Z(\lambda) = Y(\lambda)^*$, which is polynomial in $\lambda$. Differentiating $Z(\lambda) Q(\lambda) = 0$ gives \[
0 = Z^\prime(\lambda_0) P(\lambda_0)x  + Z(\lambda_0) P^\prime(\lambda_0) x = Y^* P^\prime(\lambda_0) x.
\]
Using this relation and the definition of $Y_L$ from~\eqref{eq:LeftC1orth} gives
\begin{align*}
  \Pi_\perp \begin{bmatrix}
P'(\lambda_0)x \\
0
\end{bmatrix} &=
\begin{bmatrix}
P'(\lambda_0)x \\
0
\end{bmatrix} - Y_L S^* Y^* \mathcal P(\lambda_0) \begin{bmatrix}
P'(\lambda_0)x \\
0
\end{bmatrix} \\
&= \begin{bmatrix}
P'(\lambda_0)x \\
0
\end{bmatrix} - Y_L S^* Y^*P'(\lambda_0)x = \begin{bmatrix}
P'(\lambda_0)x \\
0
\end{bmatrix}.
\end{align*}
Inserting this relation into~\eqref{eq:expressionlincondnumber} and exploiting the structre of $\tilde y_L$ finally gives
\[
 y_L^*C'_1(\lambda_0)x_L = \frac{1}{\beta\|\Lambda_0^{m-1}\|_2} y^*P'(\lambda_0)x.
\]
Inserting this relation into $\gamma_P / \gamma_{C_1} = |y^*P'(\lambda_0)x| / |y_L^*C'_1(\lambda_0)x_L| \cdot \|\Lambda_0^1\|_2 / \|\Lambda_0^m\|_2$  completes the proof.
\end{proof}}%

\change{The presence of the factor $\beta$ in the result of Lemma~\ref{t:FCFratiogeneral} makes it difficult to derive simple and insightful bounds for general polynomials. We will therefore restrict the following considerations to a quadratic polynomial $Q(\lambda) = \lambda^2 M + \lambda C + K$. Moreover, we assume that}
 $\|M\|_2 = \|K\|_2 = 1$, which can always be achieved by appropriate scaling~\cite{Fan2004,zeng2014backward}.
 \change{The following theorem shows that the linearization $C_1(\lambda)$ does not increase the weak condition number significantly when
 $|\lambda_0|\geq 1$ or $|\lambda_0|< 1$ and $\|C\|_2 \leq 1$. In the missing case, when $|\lambda_0|$ is small and $\|C\|_2$ is large, it is preferable to use a different linearization, which will be discussed below.}
  \begin{theorem}\label{t:FCFratio}
    Let $\lambda_0$ be a simple eigenvalue of an $n \times n$ quadratic eigenvalue problem $Q(\lambda) = \lambda^2 M + \lambda C + K$ with normal rank $r<n$ and $\|M\|_2 = \|K\|_2 = 1$. Let $C_1(\lambda)$ be the first companion form linearization of $Q(\lambda)$. Then, for $\delta \leq \frac{n-r}{8n^2}$ there exist constants $c_1 \approx 1.46, c_2 \approx 1.08$ such that
\[
        \frac{\kappa^{(C_1)}_w(\delta)}{\kappa^{(Q)}_w(\delta)} \leq \begin{cases}
        c_1 & \text{if } \change{|\lambda_0|< 1}, \|C\|_2 \leq 1,\\
        c_2 & \text{if } |\lambda_0|\geq 1.
        \end{cases}
\]
\end{theorem}
\begin{proof}
\change{For $m = 2$, the result of Lemma~\ref{t:FCFratiogeneral} reads
\begin{equation} 
 \label{eq:fracgamma}
  \frac{\gamma_Q}{\gamma_{C_1}} = \frac{\beta \|\Lambda_0^1\|^2_2}{\|\Lambda_0^2\|_2} = \frac{\beta \sqrt{1+|\lambda_0|^2}}{\sqrt{1+|\lambda_0|^2+|\lambda_0|^4}},
\end{equation}}%
Because of $\|y\|_2 = 1$, the \change{factor $\beta$} satisfies
\begin{eqnarray} 
\beta &=& \left\| \Pi_{\perp} \begin{bmatrix}
        y\\
        (\lambda_0M+C)^*y
\end{bmatrix}\right\|_2 \le  \left\| \begin{bmatrix}
        y\\
        (\lambda_0M+C)^*y
\end{bmatrix}\right\|_2  = 
\left\| \begin{bmatrix}
        y\\
        - \bar \lambda_0^{-1} K^* y
\end{bmatrix}\right\|_2 \nonumber \\
&\le& \min \left\{ \sqrt{1+(|\lambda_0|+\|C\|_2)^2}, \sqrt{1+|\lambda_0|^{-2}} \right\}, \label{eq:minimum}
\end{eqnarray}
where we assume, for the moment, that $\lambda_0\not=0$.

To bound $\gamma_Q/\gamma_{C_1}$, we first note that the factor $\frac{1+|\lambda_0|^2}{\sqrt{1+|\lambda_0|^2+|\lambda_0|^4}}$ in~\eqref{eq:fracgamma} is bounded by $2/\sqrt{3}$ because its maximum is attained at $|\lambda_0| = 1$. When assuming $|\lambda_0|\geq 1$,~\eqref{eq:fracgamma}  yields
    \[
        \frac{\gamma_Q}{\gamma_{C_1}} 
        \leq \frac{2}{\sqrt{3}} \sqrt{1+|\lambda_0|^{-2}} \leq \frac{2\sqrt{2}}{\sqrt{3}} \approx 1.63.
    \]
Assume now that \change{$|\lambda_0|< 1$} and $\|C\|_2 \leq 1$. As the first term in the minimum~\eqref{eq:minimum} is increasing and the other is decreasing with respect to $|\lambda_0|$, the maximum of~\eqref{eq:minimum} is attained when both terms are equal, that is, when 
$|\lambda_0| = \frac12(-\|C\|_2+\sqrt{\|C\|_2^2+4})$. In turn,
$\beta \le \sqrt{ (5+\sqrt{5})/2}$, which also holds when $\lambda_0 = 0$. Inserted into~\eqref{eq:fracgamma} this gives
\[
    \frac{\gamma_Q}{\gamma_{C_1}} = \frac{\sqrt{5+\sqrt{5}} \sqrt{2} }{\sqrt{3}} \approx 2.2. 
\]

    Using the upper bound \eqref{eq:deltaWeakCondUpperBound} for the numerator together with
    $\delta \leq \frac{n-r}{8n^2}$
    and the lower bound \eqref{eq:deltaWeakCondLowerBound} for the denominator we get
    \begin{eqnarray*}
    \frac{\kappa^{(C_1)}_w(\delta)}{\kappa^{(Q)}_w(\delta)}
         &\leq& \frac{\sqrt{\frac{n-r}{8n^2 \delta}}}{{\sqrt{\frac{(3n^2-1)(n-r)}{(3n^2+n-r-2)(3n^2+n-r-1)\delta}}}}\cdot \frac{\gamma_Q}{\gamma_{C_1}} \\
         &=& \sqrt{\frac{3n^2+n-r-2}{8n^2}}\cdot \sqrt{\frac{3n^2+n-r-1}{3n^2-1}} \cdot \frac{\gamma_Q}{\gamma_{C_1}}\le \frac{ \sqrt{3 \cdot 13} }{2\cdot\sqrt{2\cdot 11}} \cdot \frac{\gamma_Q}{\gamma_{C_1}},
    \end{eqnarray*}
    where the last inequality follows from observing that the bound is maximized for $r = 0$ and $n = 2$. Together with the derived bounds for $\gamma_Q/\gamma_{C_1}$, this concludes the proof.
\end{proof}

\change{To cover the missing case of Theorem~\ref{t:FCFratio}, we consider}
\begin{equation}\label{alternateLin}
    \hat{C}_1(\lambda) = \lambda \begin{bmatrix}
        M & C\\
        0 & I_n
    \end{bmatrix} + \begin{bmatrix}
        0 & K\\
        -I_n & 0
    \end{bmatrix},
\end{equation}
which is a strong linearization in $\mathbb L_1$ of $Q$ according to~\cite[Theorem 4.1]{de2009linearizations}. 
\begin{theorem}\label{t:ACFratio}
Under the assumptions stated in Theorem~\ref{t:FCFratio}, the linearization $\hat C_1$ defined in~\eqref{alternateLin} satisfies
\[
        \frac{\kappa^{(\hat C_1)}_w(\delta)}{\kappa^{(Q)}_w(\delta)} \leq \begin{cases}
        c_1 & \text{if } \change{|\lambda_0|> 1}, \; \|C\|_2 \leq 1, \\
        c_2 & \text{if } |\lambda_0|\leq 1.
        \end{cases}
\]
\end{theorem}
\begin{proof}
The proof is nearly identical with the \change{proofs of Lemma~\ref{t:FCFratiogeneral} and} Theorem~\ref{t:FCFratio}; in particular the relation~\eqref{eq:linrelation} is satisfied by $\hat C_1$ as well. The only significant difference is that, instead of~\eqref{eq:LeftC1}, the corresponding basis of $\ker(C_1(\lambda_0)^*)$ now takes the form
\[
 \begin{bmatrix}
   \tilde Y_L & \tilde y_L
   \end{bmatrix} = 
 \begin{bmatrix} Y & y \\
    \bar \lambda_0 M^* Y  & \bar \lambda_0 M ^* y
\end{bmatrix}.
\]
In turn, the upper bound~\eqref{eq:minimum} on $\beta$ is replaced by
\[
\beta \le \min \left\{ \sqrt{1+(|\lambda_0|^{-1}+\|C\|_2)^2}, \sqrt{1+|\lambda_0|^{2}} \right\}.
\]
Then the arguments in the rest of the proof of Theorem~\ref{t:FCFratio} apply with $|\lambda_0|$ replaced by $|\lambda_0|^{-1}$. \end{proof}

To summarize, after ensuring $\|M\|_2 = \|K\|_2 = 1$ with appropriate scaling, the use of linearization $C_1$ for eigenvalues of large magnitude and 
linearization $\hat C_1$ for eigenvalues of small magnitude ensures that the eigenvalue sensitivity remains essentially unchanged.

\section{Computation of simple finite eigenvalues}

Random perturbations turn, with probability one, a singular polynomial eigenvalue problem $P(\lambda)$ into a regular one and, at the same time, we have seen that their impact on the simple eigenvalues of $P$ is controlled, with high probability, by the weak condition numbers. These considerations suggest to use controlled random perturbations that dominate roundoff error in order \change{to} facilitate standard eigenvalue solvers.
However, attention needs to be paid to the fact that the perturbation of the singular part of $P$ will cause additional spurious eigenvalues. In summary, we suggest the following algorithm template for computing all finite simple eigenvalues of $P$:
\begin{itemize}
    \item[1.] Choose \change{a} uniformly distributed random perturbation $E(\lambda)$ and define \change{the} perturbed polynomial $\tilde{P}(\lambda):=P(\lambda) + \varepsilon E(\lambda)$.
    \item[2.] Solve \change{the} eigenvalue problem for $\tilde{P}$ by a standard method (e.g., linearization and QZ algorithm).
    \item[3.] Decide which computed eigenvalues correspond to true eigenvalues of the original matrix polynomial $P$ and which are spurious.
\end{itemize}
In the next section, we will develop a strategy for implementing Step 3.

\subsection{Criteria for eigenvalue classification} \label{sec:boundcond}

To distinguish true eigenvalues from spurious eigenvalues, we will make use of the quantity 
$
    \gamma_P 
$
featuring prominently in bounds for the weak eigenvalue condition number. For this purpose, it is necessary to not only compute the (perturbed) eigenvalue $\lambda(\varepsilon)$ in Step 2 above, but also the corresponding eigenvectors $x(\varepsilon)$, $y(\varepsilon)$.

In this section, we will show that the quantity $\gamma_P$ defined in~\eqref{eq:gammaP}, which features prominently in bounds for the weak eigenvalue condition number, provides a well suited criterion for making the decision in Step 3.

Let $y(\varepsilon)$ and $x(\varepsilon)$ be \change{suitably normalized right/left} eigenvectors corresponding to the eigenvalue $\lambda(\varepsilon)$ \eqref{eq:Expansion} of \eqref{eq:pertPEP}. It is proven in \cite{de2010first} that generically, the limits $\lim_{\varepsilon\to 0}x(\varepsilon) = \overline{x}$ and $\lim_{\varepsilon\to 0}y(\varepsilon) = \overline{y}$ belong to \change{$\ker \left(P(\lambda_0)\right)$} and \change{$\ker \left(P(\lambda_0)\right)^*$}, respectively. More precisely, the following theorem holds.
\begin{theorem} [\cite{de2010first}] \label{thm:deteran}
    Let $P(\lambda)\in\mathbb{C}^{n\times n}_m[\lambda]$ be a matrix polynomial of \change{normal} rank $r$, \change{having a simple eigenvalue $\lambda_0$ with right/left eigenvectors $x,y$}. Let  $\begin{bmatrix}
    X & x
    \end{bmatrix}$ and $\begin{bmatrix}
    Y & y
    \end{bmatrix}$ be bases for \change{$\ker \left(P(\lambda_0)\right)$} and \change{$\ker \left(P(\lambda_0)\right)^*$}, respectively. Let $E(\lambda)\in\mathbb{C}^{n\times n}_m[\lambda]$ be such that $\begin{bmatrix}
    Y & y
    \end{bmatrix}^*E(\lambda_0)\begin{bmatrix}
    X & x
    \end{bmatrix}$ is non--singular. Let $\zeta$ be an eigenvalue of the non--singular pencil
    \begin{equation}\label{pencilEvec}
        \begin{bmatrix}
    Y & y
    \end{bmatrix}^*E(\lambda_0)\begin{bmatrix}
    X & x
    \end{bmatrix} + \zeta \change{\cdot} \begin{bmatrix}
    Y & y
    \end{bmatrix}^*P'(\lambda_0)\begin{bmatrix}
    X & x
    \end{bmatrix},
    \end{equation}
    and let $a,b$ be corresponding left/right eigenvectors. \change{Then, for sufficiently small $\varepsilon>0$, the eigenvalue $\lambda(\varepsilon)$ of the perturbed matrix polynomial $P(\lambda)+\varepsilon E(\lambda)$ satisfying  \eqref{eq:Expansion}} has left and right eigenvectors \change{of the form}
\[
        y(\varepsilon) = \begin{bmatrix}
    Y & y
    \end{bmatrix}a + O(\varepsilon),\;\; x(\varepsilon) = \begin{bmatrix}
    X & x
    \end{bmatrix}b + O(\varepsilon).
\]
\end{theorem}

\change{From now on, we assume that 
$\begin{bmatrix}
    X & x
    \end{bmatrix}$, $\begin{bmatrix}
    Y & y
    \end{bmatrix}$ are orthonormal and $a,b$ have norm $1$.}
\change{By Theorem~\ref{thm:deteran}, we may assume for sufficiently small $\varepsilon$} that the eigenvectors of the perturbed pencil are approximately given by $\overline{x} = \begin{bmatrix}
    X & x
    \end{bmatrix}b$ and $\overline{y} = \begin{bmatrix}
    Y & y
    \end{bmatrix}a$. This allows us to \change{approximately} compute the following quantity:
\begin{equation}\label{overlineGK}
    \overline{\gamma_P} := |\overline{y}^*P'(\lambda_0)\overline{x} | \Big( \sum^m_{j=0}|\lambda_0|^{2j} \Big)^{-1/2}.
\end{equation}
If $\overline{x}=x$ and $\overline{y} = y$, we have $\overline{\gamma}_P = \gamma_P$. Otherwise,
\begin{align}\label{LambdaPCon}
\begin{split}
    \overline{y}^*P'(\lambda_0)\overline{x} &= a^*\begin{bmatrix}
        Y & y
    \end{bmatrix}^*P'(\lambda_0)\begin{bmatrix}
        X & x
    \end{bmatrix}b\\
    &= (y^*P'(\lambda_0)x) a^* e_{n-r+1}e^*_{n-r+1}b\\
    &= (y^*P'(\lambda_0)x)\change{\cdot \overline{a_{n-r+1}}\cdot b_{n-r+1}};
\end{split}
\end{align}
\change{see~\cite[Eq. (16)]{lotz2020wilkinson} for the second equality.}
\change{Proposition in 6.5 in \cite{lotz2020wilkinson} establishes the following result by showing
that $a$ and $b$ are uniformly (but not independently) distributed.}
\begin{proposition} \label{p:gammaP}
    Let $P(\lambda)\in\mathbb{C}^{n\times n}_m[\lambda]$ be a matrix polynomial of normal rank $r<n$ with simple eigenvalue $\lambda_0\in\mathbb{C}$, and let $E \sim \mathcal{U}(2N)$ be a random perturbation. Let $a,b$ be left and right eigenvectors of the linear pencil \eqref{pencilEvec}, let $\overline{x} = \begin{bmatrix}
    X & x
    \end{bmatrix}b$ and $\overline{y} = \begin{bmatrix}
    Y & y
    \end{bmatrix}a$, and define $\overline{\gamma}_P$ as in \eqref{overlineGK}. Then
\[
        \mathbb{E}[\overline{\gamma}_P]\leq (n-r)^{-1/2}\gamma_P, \;\; \mathbb{P}(\overline{\gamma}_P\geq \gamma_P\cdot t)\leq e^{-(n-r)t^2}.
\]
\end{proposition}
Proposition \ref{p:gammaP} implies that $\overline{\gamma}_P$ tends to underestimate $\gamma_P$ by a factor $1/\sqrt{n-r}$, reflecting that worst-case perturbation directions represent a rather specific set of rank-one matrices~\cite[Theorem 5]{tisseur2000backward}. From~\eqref{LambdaPCon} it also follows that $\overline{\gamma}_P \le \gamma_P$ and hence an ill-conditioned eigenvalue ($\gamma_P \approx 0$) never turns into a well-conditioned eigenvalue by a sufficiently small perturbation. For practical purposes, the opposite direction is somewhat more important, that is, any well-conditioned eigenvalue $\lambda_0$ of the original problem will show up as a well-conditioned eigenvalue of the perturbed problem as well. \change{The following proposition establishes this property: A well-conditioned eigenvalue ($\gamma_P \gg 0$) is unlikely to turn (asymptotically) into an ill-conditioned eigenvalue ($\overline{\gamma}_P\approx 0$).}
\change{\begin{proposition}\label{p:gammaPopposite} Considering the setting of Proposition~\ref{p:gammaP}, it holds for every $0 < t \le 1$ that
 \[
 \mathbb{P}(\overline{\gamma}_P < \gamma_P \cdot  t)\leq 2 t (n-r).
 \]
\end{proposition}}
\begin{proof}
 \change{Because $a,b$ are uniformly distributed, we have
$|\change{a_{n-r+1}|^2,|b_{n-r+1}|^2}\sim \Bdist(1,n-r)$.}     Using that
\begin{equation*}
    \min \left( \change{|a_{n-r+1}|^2, |b_{n-r+1}|^2} \right)^2 \leq \change{|a_{n-r+1}|^2\cdot |b_{n-r+1}|^2},
\end{equation*}
we then obtain
\begin{align*}
    & \mathbb{P}(\change{|a_{n-r+1}|^2\cdot |b_{n-r+1}|^2}<t^2)\\
    &\leq \mathbb{P}(\min \left( \change{|a_{n-r+1}|^2, |b_{n-r+1}|^2} \right)^2 < t^2) \\
    &\leq \mathbb{P}( \change{|a_{n-r+1}|^2}< {t}) + \mathbb{P}( \change{|b_{n-r+1}|}^2< {t})\\
    &= \frac{2}{\Bfun(1,n-r)}\int^{{t}}_{0} (1-x)^{n-r-1}\dx \le 2 t (n-r),
\end{align*}
\change{where the last step uses that $\Bfun(1,n-r) = 1/(n-r)$. Inserting this inequality into~\eqref{LambdaPCon} gives
\[
\mathbb{P}\big(|\overline{y}^*P'(\lambda_0)\overline{x}| < |{y}^*P'(\lambda_0){x}| \cdot t \big) \le 2 t (n-r),
\]
which concludes the proof by the definitions of $\gamma_P$, $\overline{\gamma}_P$.}
\end{proof}

Propositions~\ref{p:gammaP} and~\ref{p:gammaPopposite} do not exclude the possibility that the perturbed problem has spurious \emph{well-conditioned} eigenvalues originating from the singular part of $P(\lambda)$. In the following, we will argue why this cannot happen for small $\varepsilon$.

\change{Recall that a} complex number $\lambda_0$ is an eigenvalue of $P(\lambda)$ if and only if \change{$\rank (P({\lambda_0})) < \nrank \left(P\right) = r$} or, equivalently, $\sigma_r(P({\lambda_0})) = 0$, where $\sigma_j(\cdot)$ denotes the $j$th largest singular value of a matrix. \change{Therefore it seems reasonable to classify an eigenvalue $\tilde{\lambda}$ of the perturbed problem $\tilde P(\lambda) = P(\lambda)+\varepsilon E(\lambda)$ as \emph{spurious} if $\sigma_r(P(\tilde{\lambda}))$ is not very small.}
\change{More specifically, we consider $\tilde \lambda$ spurious if} \begin{equation} \label{eq:condspurious}
\sigma_r(P(\tilde{\lambda})) \ge 5  \change{\|E(\tilde \lambda)\|_F} \varepsilon.
\end{equation}
\change{We measure the conditioning of $\tilde \lambda$ using 
\[
 \overline{\kappa}_{\tilde P} = \overline{\gamma}^{-1}_{\tilde P}, \quad \overline{\gamma}_{\tilde P} = \frac{|\tilde y^*\tilde P'(\tilde \lambda)\tilde x|}{\|\tilde \Lambda^m \|_2},\quad \tilde \Lambda^m = \big[1,\tilde \lambda,\ldots,\tilde \lambda^{m} \big].
\]
where $\tilde{x}, \tilde{y}$ denote right/left eigenvectors belonging to $\tilde \lambda$ with $\|\tilde x\|_2 = \|\tilde y\|_2 = 1$. 
Generically, $\tilde P(\lambda)$ is non-singular and $\tilde \lambda$ is simple, in which case $\overline{\kappa}_{\tilde P}$ coincides with a common definition~\cite[Eq.~(10)]{adhikari2011structured} of eigenvalue condition number. The following result shows that $\overline{\kappa}_{\tilde P} = O(\varepsilon^{-1})$ for a spurious eigenvalue. }

\change{\begin{lemma}\label{l:PgammaSpourious}
With the notation introduced above, it holds that
\[
 \overline{\gamma}_{\tilde P} \le \varepsilon \|E\| (12 / \sigma_r(P(\tilde{\lambda})) + m)
\]
\end{lemma}
\begin{proof} For every $\varepsilon \ge 0$, let  $\tilde X(\varepsilon), \tilde Y(\varepsilon) \in \mathbb C^{n\times (n-r)}$ contain the right/left singular vectors belonging to the smallest $n-r$ singular values of $\tilde P(\tilde{\lambda})$.
Because of~\eqref{eq:condspurious}, $P(\tilde \lambda)$ has rank $r$ and, in turn, $\tilde X(0)$ and $\tilde Y(0)$ are orthonormal bases of $\ker(P(\tilde \lambda))$ and $\ker(P(\tilde \lambda)^*)$, respectively. 

  Existing perturbation theory for singular vectors~\cite{Stewart1973,wedin1927perturbation} allows us to link $\tilde X(\varepsilon),\tilde Y(\varepsilon)$ with $\tilde X(0),\tilde Y(0)$. In particular, Theorem 8.6.5 from~\cite{Golub2013} applies because of~\eqref{eq:condspurious} and 
yields the existence of matrices $Q_X,Q_Y$ such that
\[
 \mathrm{range}(\tilde X(\varepsilon)) = \mathrm{range}(\tilde X(0)+Q_X),\quad 
 \mathrm{range}(\tilde Y(\varepsilon)) = \mathrm{range}(\tilde Y(0)+Q_Y),
\]
and $\|[Q_X,Q_Y]\|_F \le 4 \varepsilon   \|E(\tilde \lambda)\|_F / \tau$, where we set $\tau := \sigma_r(P(\tilde{\lambda}))$.
By definition, $\mathrm{range}(\tilde X(\varepsilon))$ and
$\mathrm{range}(\tilde Y(\varepsilon))$ contain 
$\tilde{x}$ and $\tilde{y}$, respectively. Hence, there are vectors $\tilde b, \tilde a$ such that 
$\tilde x = \tilde X(\varepsilon) \tilde b$, $\tilde y = \tilde Y(\varepsilon) \tilde a$, and
$\|\tilde a\|_2 = \|\tilde b\|_2 = 1$. Setting 
$q_x = Q_X \tilde b$, $q_y = Q_Y \tilde a$, we thus obtain
\begin{equation} \label{eq:formulaxy}
     \tilde{x} = \tilde X(0) \tilde{b} + q_x, \quad \tilde{y} = \tilde Y(0) \tilde{a} + q_y,\quad  \max\{\|q_x\|_2,\|q_y\|_2\}\leq 4 \varepsilon \|E(\tilde \lambda)\|_F  / \tau < 1.
\end{equation}
To proceed from here, we choose (polynomial) bases $X(\lambda)$, $Y(\lambda)$ of the right/left nullspaces of $P$ such that 
$\tilde X(0) = X(\tilde \lambda)$, $\tilde Y(0) = Y(\tilde \lambda)$.
Differentiating the relation $P(\lambda) X(\lambda) = 0$ gives $P'(\tilde \lambda) X(\tilde \lambda) + P(\tilde \lambda) X'(\tilde \lambda) = 0$ and, hence,
\[
 \tilde Y(0)^* P^\prime(\tilde \lambda) \tilde X(0) = Y(\tilde \lambda)^* P'(\tilde \lambda) X(\tilde \lambda) = 0.
\]
Combined with~\eqref{eq:formulaxy}, we obtain
\begin{align*}
 |\tilde y^* P'(\tilde \lambda) \tilde x| &\le \|P'(\tilde \lambda)q_x\|_2 + \|q_y^* P'(\tilde \lambda) \|_2 + |q_y^* P'(\tilde \lambda) q_x| \\
 &\le \|P'(\tilde \lambda)\|_2 (\|q_x\|_2 + \|q_y\|_2 + \|q_x\|_2 \|q_y\|_2) \\
 &< 12 \varepsilon \|E(\tilde \lambda)\|_F  / \tau \le 12 \varepsilon \|\tilde \Lambda^m\|_2 \|E\| / \tau.
\end{align*}
In turn,
\[
 |\tilde y^* \tilde P'(\tilde \lambda) \tilde x| \le |\tilde y^* P'(\tilde \lambda) \tilde x| + \varepsilon  \|E^\prime(\tilde \lambda)\|_F
 \le \varepsilon \|\tilde \Lambda^m\|_2 \|E\| (12/\tau + m),
\]
where we used $
\|E^\prime(\tilde \lambda)\|_F \le \|E\| \big( \sum_{j = 1}^m | j \tilde \lambda^{j-1}|^2 \big)^{1/2} \le m \|E\| \|\tilde \Lambda^m\|_2$. By the definition of $\overline{\gamma}_{\tilde P}$, this concludes the proof.
\end{proof}}

Propositions~\ref{p:gammaP},~\ref{p:gammaPopposite} and Lemma~\ref{l:PgammaSpourious} justify the use of the condition number for detecting all (well-conditioned) eigenvalues of the original problem. If the \change{condition number} is modest then it is likely that a (well-conditioned) eigenvalue of the original problem has been detected, while a high \change{condition number} indicates an ill-conditioned or spurious eigenvalue.

\subsection{Algorithm for singular generalized eigenvalue problems}

The results from Section~\ref{sec:boundcond} lead to Algorithm~\ref{a:mainAlg} for computing finite well-conditioned eigenvalues of a singular matrix pencil \change{$\lambda A_1+A_0$} by perturbing it with a random pencil \change{$E(\lambda) = \lambda E_1 + E_0$}. 

The algorithm for solving the singular generalized eigenvalue problem by using the condition number of computed eigenvalues as the criteria for the detection of true eigenvalues is the following:
\begin{algorithm}[H]
\caption{Algorithm for solving singular generalized eigenvalue problem}
\begin{algorithmic}[1]
\label{a:mainAlg}
\normalsize
\REQUIRE Singular pencil \change{$\lambda A_1+A_0$} of order $n$, $\varepsilon$, $\delta$, $\mathsf{tol}$
\STATE $\change{E_{0}} = \texttt{randn}(n,n)+i\cdot\texttt{randn}(n,n)$, $\change{E_{1}} = \texttt{randn}(n,n)+i\cdot\texttt{randn}(n,n)$
\STATE \change{$E_0 = E_0/\|E_0\|$, $E_1 = E_1/\|E_1\|$}
\STATE \change{$\lambda \tilde{A}_1 + \tilde{A}_0 = \lambda A_1+ A_0 + \varepsilon(\lambda E_1 + E_0)$}
\STATE Compute eigenvalues $\lambda_{j}$ and corresponding eigenvectors $x_{j}$ and $y_{j}$ of pencil \change{$\lambda \tilde{A}_1+\tilde{A}_0$}
\STATE For each eigentriple $(\lambda_{j},x_{j},y_{j})$ compute $\overline{\kappa}_j=\frac{\sqrt{1+|\lambda_j|^2}}{|y_j^*Bx_j|}$
\STATE $\Lambda = \left\{ \lambda_{j}: \; \change{\overline{\kappa}_j}\leq \mathsf{tol} \right\}$
\RETURN $\Lambda$
\end{algorithmic}
\end{algorithm}

\subsection{Algorithm for singular quadratic eigenvalue problems}

In this section we propose a new procedure to compute finite well-conditioned eigenvalues of a singular $n\times n$ quadratic eigenvalue problem $Q(\lambda) = \lambda^2M+\lambda C +K$.

The first step is to obtain a normalized problem $\breve{Q}(\lambda) = \lambda^2\breve{M} + \lambda\breve{C} + \breve{K}$ with $\|\breve{M}\|_2 = \|\breve{K}\|_2=1$ by using the scaling parameters 
 $\gamma = \sqrt{\|K\|_2 / \|M\|_2}$, $\omega = 1/\|K\|_2$
 and setting
 \[
     \breve{M} = \omega\gamma^2M, \;\; \breve{C} = \omega\gamma C, \;\; \breve{K} = \omega K.
 \]
  Then small random perturbations of the coefficient matrices are used to obtain a regular quadratic eigenvalue problem $\tilde{Q}(\lambda) = \lambda^2\tilde{M}+\lambda \tilde{C} + \tilde{K}$. We apply the QZ algorithm to the linearizations $C_1(\lambda)$ and $\hat{C}_1(\lambda)$ of the perturbed problem  in order to compute eigenvalues of $\tilde Q$ and extract the corresponding right and left eigenvectors. Again, as in Algorithm \ref{a:mainAlg}, we use the condition number to classify the computed eigenvalues.

Since Theorems \ref{t:FCFratio} and \ref{t:ACFratio} hold, we will look for eigenvalues of magnitude larger than one among those computed using $C_1(\lambda)$, and those with magnitude smaller than one among those computed using $\hat{C}_1(\lambda)$. After we detect all eigenvalue of $\tilde Q$, we need to rescale them. This procedure is summarized in Algorithm~\ref{a:mainAlgQEP}.
\begin{remark}[Eigenvector recovery]
    After computing eigenvalues and corresponding left and right eigenvectors of linearizations $C_1(\lambda)$ and $\hat{C}_2(\lambda)$ in lines 6 and 7 of Algorithm \ref{a:mainAlgQEP}, we recover eigenvectors for the regular perturbed quadratic eigenvalue problem as in \cite{zeng2014backward}. More precisely, if linearization $C_1(\lambda)$ is used to compute eigentriple $(\lambda_j,x_j,y_j)$ left and right eigenvector for the quadratic problem is recovered as $y_j(1:n)$ and $x_j(1:n)$, respectively.
    
    In the case of linearization $\hat{C}_1(\lambda)$ and computed eigentriple $(\hat{\lambda}_j,\hat{x}_j,\hat{y}_j)$, left and right eigenvectors are recovered as $\hat{y}_j(1:n)$ and $\hat{x}_j(n+1:2n)$, respectively. 
    
    This recovery insures backward stability.
\end{remark}
\begin{algorithm}[H]
\caption{Algorithm for solving singular quadratic eigenvalue problem} 
\begin{algorithmic}[1]
\label{a:mainAlgQEP}
\normalsize
\REQUIRE Singular quadratic matrix polynomial $\lambda^2M+\lambda C + K$ of order $n$, $\varepsilon$, $\delta$, $\mathsf{tol}$
\STATE $\gamma = \sqrt{\|K\|/\|M\|}$, $\omega = 1/\|K\|$
\STATE $\breve{M} = \omega\gamma^2M$, $\breve{C} = \omega\gamma C$, $\breve{K} = \omega K$
\STATE $E_{j} = \texttt{randn}(n,n)+i\cdot\texttt{randn}(n,n)$, $E_j = E_j/\|E_j\|$, $j=1,2,3$
\STATE $\tilde{M} = \breve{M}+\varepsilon E_1$, $\tilde{C} = \breve{C}+\varepsilon E_2$, $\tilde{K} = \breve{K}+\varepsilon E_3$
\STATE $ C_1(\lambda) = \begin{bmatrix} \tilde{C} & \tilde{K}\\ -I_n & 0 \end{bmatrix} - \lambda \begin{bmatrix} -\tilde{M} & 0\\ 0 & -I_n \end{bmatrix} $, $\hat{C}_1(\lambda) =\begin{bmatrix}
        0 & \tilde{K}\\
        -I_n & 0
    \end{bmatrix} - \lambda \begin{bmatrix}
        -\tilde{M} & -\tilde{C}\\
        0 & -I_n
    \end{bmatrix} $
\STATE Compute eigenvalues $\lambda_{j}$ and corresponding eigenvectors $x_{j}$ and $y_{j}$ of pencil $C_1(\lambda)$
\STATE Compute eigenvalues $\hat{\lambda}_{j}$ and corresponding eigenvectors $\hat{x}_{j}$ and $\hat{y}_{j}$ of pencil $\hat{C}_1(\lambda)$
\FOR{$j=1,\ldots,2n$}
\STATE For each $(\lambda_{j},x_{j}(1:n),y_{j}(1:n))$, $|\lambda_j|\geq 1$ compute $\overline{\kappa}_{j} = \frac{\sqrt{1+|\lambda_j|^2+|\lambda_j|^4}}{|y_j^*(1:n)(2\lambda_j M +C)x_j(1:n)|}$
\STATE For each $\small{(\!\hat{\lambda}_{j},\hat{x}_{j}(\!n+1:2n\!),\hat{y}_{j}(\!1:n\!)\!\!)}$,$|\!\hat{\lambda}_j\!|\!\! <\!\! 1$ compute $\hat{\overline{\kappa}}_{j}\!\! =\!\! \frac{\sqrt{1+|\!\hat{\lambda}_j|^2+|\hat{\lambda}_j\!|^4}}{|\!\hat{y}_j^*(\!1:n\!)(\!2\hat{\lambda}_j M +C\!)\hat{x}_j(\!n+1:2n\!)\!|}$
\ENDFOR
\STATE $\Lambda = \left\{ \gamma\cdot \lambda_{j}: \; \overline{\kappa}_{j}\leq \mathsf{tol} \right\} \cup \left\{ \gamma\cdot\hat{\lambda_{j}}: \; \hat{\overline{\kappa}}_{j}\leq \mathsf{tol} \right\}$
\RETURN $\Lambda$
\end{algorithmic}
\end{algorithm}

\section{Numerical examples}
In this section we present the performance of the proposed Algorithms \ref{a:mainAlg} and \ref{a:mainAlgQEP}. \change{We use MATLAB 9.0.0.341360 (R2016a) in double precision (IEEE Standard 754).} For singular quadratic eigenvalue problems we use several examples from the literature and other examples constructed by ourselves. In the case of matrix pencils, we use three examples from \cite{hochstenbach2019solving}.

Let $n$ be the dimension of the problem considered, generalized or quadratic eigenvalue problem. If not stated otherwise, the following default values are used in the numerical experiments
\begin{align}\label{defaultpar}
\begin{split}
    \varepsilon = 10^{-8},\quad \mathsf{tol} = 10^4.
\end{split}
\end{align}

Since we know the number of finite simple eigenvalue and their values for all examples, we will present the empirical probability that Algorithm~\ref{a:mainAlg} or~\ref{a:mainAlgQEP} detects all of them. \change{This is done in the following way. First we check if the number of the detected eigenvalues equals the number of true finite eigenvalues of the original problem. Then, in order to check that these eigenvalues are not spurious, we compute $\sigma_r(Q(\lambda_0))$, where $r = \nrank(Q(\lambda))$, for every $\lambda_0$ that is declared as an eigenvalue by the algorithm. If this value is smaller than $100 \varepsilon\cdot \max(1,|\lambda_0|^2)$ we conclude that detected eigenvalue is not spurious. Analogously, for generalized eigenvalue problems we compute $\sigma_r(\lambda_0A_1+A_0)$, where $r=\nrank(\lambda A_1+A_0)$.} We perform $n_t = 1000$ runs of the algorithm with different perturbation matrices and count how many times $n_s$ it was successful. Then, the empirical probability that the algorithm finds all simple finite eigenvalues of a given singular problem is computed as $p = n_s/n_t$.

\paragraph{Example 1.}
This is the singular quadratic eigenvalue problem from Example 1 in \cite{van1983eigenstructure}:
\begin{equation*}
    \lambda^2 \begin{bmatrix}
        1 & 4 & 2\\
        0 & 0 & 0\\
        1 & 4 & 2
    \end{bmatrix} + \lambda \begin{bmatrix}
        1 & 3 & 0\\
        1 & 4 & 2\\
        0 & -1 & -2
    \end{bmatrix} + \begin{bmatrix}
        1 & 2 & -2\\
        0 & -1 & -2\\
        0 & 0 & 0
    \end{bmatrix}.
\end{equation*}
It is known that $\change{\nrank (Q)} = 2$, and that there is only one finite eigenvalue $\lambda_0 =1$. The empirical probability that Algorithm \ref{a:mainAlgQEP} finds this eigenvalue is $p=0.999$.
\paragraph{Example 2.}
This is Example 1.8 from \cite{byers2008trimmed}:
\begin{equation*}
    \lambda^2 \begin{bmatrix}
        1 & 0\\
        0 & 0
    \end{bmatrix} + \lambda \begin{bmatrix}
        1 & 0\\
        0 & 0
    \end{bmatrix} + \begin{bmatrix}
        0 & 0\\
        1 & 0
    \end{bmatrix}.
\end{equation*}
It is known that there is no finite eigenvalue. The empirical probability that Algorithm \ref{a:mainAlgQEP} correctly detects no finite eigenvalue is one.
\paragraph{Example 3.}
This is Example 2.5 from \cite{byers2008trimmed}:
\begin{equation*}
    \lambda^2\begin{bmatrix}
        1 & 0 & 0 & 0\\
        0 & 1 & 0 & 0\\
        0 & 0 & 0 & 0\\
        0 & 0 & 0 & 0
    \end{bmatrix} + \lambda \begin{bmatrix}
        0 & 1 & 1 & 0\\
        1 & 0 & 0 & 1\\
        1 & 0 & 0 & 0\\
        0 & 0 & 0 & 0
    \end{bmatrix} + \begin{bmatrix}
        0 & 0 & 0 & 0\\
        0 & 0 & 1 & 0\\
        0 & 1 & 0 & 1\\
        0 & 0 & 0 & 0
    \end{bmatrix}.
\end{equation*}
It is known that there is one finite eigenvalue zero, and two infinite eigenvalues. The empirical probability that Algorithm \ref{a:mainAlgQEP} detects the zero eigenvalue is $p = 1$.
\paragraph{Example 4.}
The following example was constructed so that $1$ and $2$ are only finite eigenvalues
\begin{equation*}
    \lambda^2 \begin{bmatrix}
        0 & 1 & 0\\0 & 0 & 1\\0 & 1 & 1
    \end{bmatrix} + \lambda \begin{bmatrix}
        1 & -1 & 0\\0 & 1 & -2\\1 & 0 & -2
    \end{bmatrix} + \begin{bmatrix}
        -1 & 0 & 0\\0 & -2 & 0\\-1 & -2 & 0
    \end{bmatrix}.
\end{equation*}
The empirical probability of Algorithm \ref{a:mainAlgQEP} detecting these eigenvalues is $p = 0.999$.
\paragraph{Example 5.}
This example was constructed so that it has 5 finite eigenvalues $\lambda_i = 1+1e-5\cdot i$, $i=1,\ldots,5$. The coefficient matrices are of size $8\times 8$ and defined as
\begin{align*}
    M(i,j) &= \begin{cases}
    1 & \text{if } i=1,\ldots 5,\;\; j=i+1,\\
    0 & \text{otherwise}
    \end{cases},\;\; M \change{\gets} U^*MV,\\
    C(i,j) & = \begin{cases}
    1 & \text{if }i=1,\ldots 5, \;\;j=i,\\
    -\lambda_i & \text{if }i=1,\ldots 5, \;\;j=i+1\\
    0 & \text{otherwise}
    \end{cases},\;\; C \change{\gets} U^*CV,\\
    K &= \diag(-\lambda_i,0_3),\;\;K \change{\gets} U^*KV,
\end{align*}
with $U = \texttt{orth(rand(8,8))}$ and $V = \texttt{orth(rand(8,8))}$. The empirical probability that Algorithm \ref{a:mainAlgQEP} correctly detects $5$ eigenvalues is $p = 0.999$.
\paragraph{Example 6.}
This example was constructed so that it has 8 finite eigenvalues $\lambda_1 = 0$, $\lambda_i = 1/i$, $i=2,\ldots,8$. The coefficient matrices are of order $11\times 11$ and defined as
\begin{align*}
    M(i,j) &= \begin{cases}
    1 & \text{if } i=1,\ldots 8,\;\; j=i+1,\\
    0 & \text{otherwise}
    \end{cases},\;\;M \change{\gets} U^*MV,\\
    C(i,j) & = \begin{cases}
    1 & \text{if }i=1,\ldots 8, \;\;j=i,\\
    -\lambda_i & \text{if }i=1,\ldots 8, \;\;j=i+1\\
    0 & \text{otherwise}
    \end{cases},\;\; C \change{\gets} U^*CV,\\
    K &= \diag(-\lambda_i,0_3),\;\;K \change{\gets} U^*KV,
\end{align*}
with $U = \texttt{orth(rand(11,11))}$ and $V = \texttt{orth(rand(11,11))}$. The empirical probability that Algorithm \ref{a:mainAlgQEP} correctly detects $8$ eigenvalues is $p = 0.999$.
\paragraph{Example 7.}
In this example we consider reversed quadratic eigenvalue problem from Example 6. Therefore, there are now 7 finite eigenvalues $\lambda_i = i+1$, $i=1,\ldots,7$ and one infinite eigenvalue. The empirical probability that Algorithm \ref{a:mainAlgQEP} correctly detects $7$ eigenvalues is $p = 0.991$.

\paragraph{Example 8.}
In this example, we define the diagonal matrix
\begin{equation*}
    D = \diag(1,a^2,a,1,a^3,1,a^4,a^5,a^6, 1,1),\;\; a=2.
\end{equation*}
We use the coefficient matrices from Example 7 and change them to $M \gets UDMDV$, $C \gets UDCDV$ and $K \gets UDKDV$. Diagonal scaling makes the eigenvalues more ill conditioned. In turn, the default tolerance $10^4$ does not yield satisfactory results. More precisely, the empirical probability of Algorithm \ref{a:mainAlgQEP} to detect $8$ eigenvalues with the default input parameters $\eqref{defaultpar}$ is $p = 0.527$. On the other hand, if we increase the tolerance to $\mathsf{tol} = 10^5$, the probability $p$ increases to $0.952$.\\

Next, we present the performance of Algorithm \ref{a:mainAlg} for three examples of generalized eigenvalue problems from \cite{hochstenbach2019solving}.

\paragraph{Example 9.}
This is the $7\times 7$ matrix pencil from Example 6.1. in \cite{hochstenbach2019solving}, which has normal rank $\nrank(A-\lambda B) = 6$ and two finite eigenvalues, $1/2$ and $1/3$.
The empirical probability that Algorithm \ref{a:mainAlg} correctly \change{detects} both of these eigenvalues is $p = 1$.
\paragraph{Example 10.}
This is Example C3 from \cite{demmel1988accurate}:
\begin{equation*}
    \begin{bmatrix}
        1 & -2 & 100 & 0 & 0\\
        1 & 0 & -1 & 0 & 0  \\
        0 & 0 & 0 & 1 & -75 \\
        0 & 0 & 0 & 0 & 2
    \end{bmatrix} - \lambda \begin{bmatrix}
        0 & 1 & 0 & 0 & 0\\
        0 & 0 & 1 & 0 & 0\\
        0 & 0 & 0 & 1 & 0\\
        0 & 0 & 0 & 0 & 1
    \end{bmatrix}.
\end{equation*}
This pencil is rectangular of size $4\times 5$. Therefore, before using Algorithm \ref{a:mainAlg}, we add an additional row of zeros to make it square of order $5$. \change{(Notice that the regular part of the pencil remains unchanged by adding zero row. For more details see \cite{hochstenbach2019solving}.)} It is known that this pencil has two finite eigenvalues $1$ and $2$. The empirical probability of Algorithm \ref{a:mainAlg} detecting both of these eigenvalues is $p = 0.982$.
\paragraph{Example 11.}
This is the matrix pencil constructed in Example 6.4 from \cite{hochstenbach2019solving}. It has size $300$ and a total number of $90$ finite eigenvalues. Algorithm \ref{a:mainAlg} with default input arguments \eqref{defaultpar} never detects all eigenvalues. However, \change{if we increase the tolerance to $\mathsf{tol} = 10^8$ the empirical probability of success increases to $p=0.93$.}
\begin{remark}
We compared Algorithm~\ref{a:mainAlg} with the rank--completing perturbation algorithm from \cite{hochstenbach2019solving} for matrix pencils. The latter algorithm detects all true eigenvalues in every example in this section. We can conclude that both algorithms detected all true eigenvalues for all examples with the exception for Example 11 with default parameters. By tuning the \change{tolerance} we can detect the eigenvalues with high probability. One advantage of our algorithm over the rank--completing algorithm is that we do not have to explicitly compute the normal rank of the problem.
\end{remark}

\section{Conclusion}

We analyzed the influence of two linearizations on the $\delta$--weak condition number of singular quadratic matrix polynomials. Our results nicely match existing results for the regular case. In addition, we supplemented the theory from \cite{lotz2020wilkinson} in order to derive a reliable criterion for computing and detecting finite well-conditioned eigenvalues of singular matrix polynomials. The resulting algorithms perform well for the examples considered.

\section*{Acknowledgements}
 The authors thank Bor Plestenjak for helpful discussions on the numerical solution of singular eigenvalue problems, and Petra Lazi\'{c} for discussions regarding the theory of $\delta$--weak condition numbers. \change{The authors also thank the reviewers for their careful reading and useful comments.}

\bibliography{References}
\bibliographystyle{plain} 
\end{document}